\RequirePackage{fix-cm}
\documentclass[11pt, reqno]{amsart}
\usepackage{amsmath}
\usepackage{setspace}
\usepackage{fancyhdr}
\usepackage{mathrsfs}
\usepackage{xifthen}
\usepackage{verbatim}
\usepackage{hyperref}
\usepackage[all]{xcolor}
\usepackage{braket}
\usepackage{stmaryrd}

\usepackage{float}

\usepackage{comment}

\usepackage{geometry}
\usepackage{fancyhdr}
\usepackage{textcomp}
\usepackage{amsthm}
\usepackage{amstext}
\usepackage{amsfonts}
\usepackage{amssymb}
\usepackage{fixltx2e}
\usepackage{graphicx}

\usepackage{tikz}
\usepackage{tikz-cd}
\usepackage{caption}
\usetikzlibrary{calc,backgrounds}
\usetikzlibrary{matrix,arrows,decorations.pathmorphing}

\usepackage{bigstrut}

\makeatletter

\newcommand{\schff}{X}
\newcommand{\adicff}{\mathcal{X}}
\newcommand{\Spa}{\mathrm{Spa}}
\newcommand{\Spd}{\mathrm{Spd}}
\newcommand{\HN}{\mathrm{HN}}
\newcommand{\HNvec}{\overrightarrow{\mathrm{HN}}}
\newcommand{\Surj}{\mathcal{S}\mathrm{urj}}
\newcommand{\Inj}{\mathcal{I}\mathrm{nj}}
\newcommand{\Hom}{\mathcal{H}\mathrm{om}}
\newcommand{\Aut}{\mathcal{A}\mathrm{ut}}
\newcommand{\Extdiamond}{\mathcal{E}\mathrm{xt}}

\newcommand{\rk}{\mathrm{rk}}

\newcommand{\inj}{\hookrightarrow}
\newcommand{\surj}{\twoheadrightarrow}
\newcommand{\nonneg}{{\geq 0}}

\newcommand{\mumax}{\mu_\text{max}}
\newcommand{\mumin}{\mu_\text{min}}
\newcommand{\trivbundle}{\Ocal}

\newcommand{\genred}{\overline}

\newcommand{\constgrpsch}{\underline}

\newcommand{\pseudounif}{\varpi}
\newcommand{\finext}{E}
\newcommand{\algclosedperfdfield}{F}
\newcommand{\genresfield}{k}
\newcommand{\genperfdring}{R}

\newcommand{\integerring}[1]{\Ocal_{#1}}
\newcommand{\witt}{W}
\newcommand{\fracwitt}{K_0}

\newcommand{\Perfd}{\mathrm{Perfd}}

\newcommand{\rkone}{{\mathrm{rk} 1}}

\numberwithin{equation}{section}

\usepackage{color}

\newcommand{\F}{\mathbb{F}}
\newcommand{\Q}{\mathbb{Q}}
\newcommand{\Z}{\mathbb{Z}}

\DeclareMathOperator{\coker}{coker}

\DeclareMathOperator{\rank}{rank}

\DeclareMathOperator{\Proj}{Proj\,}

\DeclareMathOperator{\dR}{dR}

\DeclareMathOperator{\im}{im}

\DeclareFontFamily{OT1}{rsfs}{}
\DeclareFontShape{OT1}{rsfs}{n}{it}{<-> rsfs10}{}
\DeclareMathAlphabet{\mathscr}{OT1}{rsfs}{n}{it}

\newcommand{\Dcal}{\mathcal{D}}
\newcommand{\Ecal}{\mathcal{E}}
\newcommand{\Fcal}{\mathcal{F}}

\newcommand{\Hcal}{\mathcal{H}}

\newcommand{\Ocal}{\mathcal{O}}

\newcommand{\Qcal}{\mathcal{Q}}

\newcommand{\Vcal}{\mathcal{V}}
\newcommand{\Wcal}{\mathcal{W}}

\newcommand{\Ycal}{\mathcal{Y}}

\newcommand{\dsm}{\oplus}

\newcommand{\customlabel}[2]{#2\def\@currentlabel{#2}\label{#1}}

\newtheorem{thm}[subsubsection]{Theorem}
\newtheorem{lemma}[subsubsection]{Lemma}
\newtheorem{prop}[subsubsection]{Proposition}
\newtheorem{cor}[subsubsection]{Corollary}

\theoremstyle{remark}
\newtheorem*{remark}{Remark}
\newtheorem{defn}[subsubsection]{Definition}

\newtheorem*{thm*}{Theorem}
\makeatletter
\def\th@remark{%
  \thm@headfont{\bfseries}%
  \normalfont 
}
\def\imod#1{\allowbreak\mkern5mu({\operator@font mod}\,\,#1)}
\makeatother

\usepackage{enumitem}		
\theoremstyle{theorem}
\newtheorem{theorem}[subsubsection]{Theorem}

\numberwithin{equation}{section}

\makeatother

\begin{document}
	
	\tikzset{
		node style sp/.style={draw,circle,minimum size=\myunit},
		node style ge/.style={circle,minimum size=\myunit},
		arrow style mul/.style={draw,sloped,midway,fill=white},
		arrow style plus/.style={midway,sloped,fill=white},
	}
    
	\title{On certain extensions of vector bundles in p-adic geometry}
   
    \author[S. Hong]{Serin Hong}
    \address{Department of Mathematics, University of Michigan, 530 Church Street, Ann Arbor MI 48109}
    \email{serinh@umich.edu}
    
    \begin{abstract} Given two arbitrary vector bundles on the Fargues-Fontaine curve, we give an explicit criterion in terms of Harder-Narasimhan polygons on whether they realize a semistable vector bundle as their extensions. Our argument is largely combinatorial and builds upon the dimension analysis of certain moduli spaces of bundle maps developed in \cite{Arizona_extvb}. 

    \end{abstract}
	
	\maketitle

	\tableofcontents
	
	\rhead{}

	\chead{}
\section{Introduction}

\subsection{The main result}\label{main results}$ $

Over the past decade, $p$-adic Hodge theory has undergone a remarkable development driven by a series of new geometric ideas. Of particular importance among such ideas are the theory of perfectoid spaces introduced by Scholze \cite{Scholze_perfectoid} and the geometric reformulation of $p$-adic Hodge theory by Fargues and Fontaine \cite{FF_curve} using a noetherian one-dimensional $\Q_p$-scheme called the \emph{Fargues-Fontaine curve}. Some notable applications of these ideas are the geometrization of the local Langlands correspondence by Fargues-Scholze \cite{FS_geomLL} and the construction of local Shimura varieties by Scholze-Weinstein \cite{SW_berkeley}.

In this article, we address the question of determining whether there exists a short exact sequence among three given vector bundles on the Fargues-Fontaine curve. This question naturally arises in the study of various objects in $p$-adic geometry. For example, a partial answer to this question obtained by the author and his collaborators in \cite{Arizona_extvb} leads to the work of Hansen \cite{Hansen_degenvb} that describes precise closure relations among the Harder-Narasimhan strata on the stack of vector bundles on the Fargues-Fontaine curve. In addition, a general answer to this question can be used to describe the geometry of the $p$-adic flag variety and the $B_{\dR}^+$-Grassmannian in terms of two natural stratifications, namely the Harder-Narasimhan stratification and the Newton stratification, in line with the work of many authors including Caraiani-Scholze \cite{CS_cohomunitaryshimura}, Chen-Fargues-Shen \cite{CFS_admlocus}, Shen \cite{Shen_HNstrata}, Chen \cite{Chen_FRconjnonbasic}, Viehmann \cite{Viehmann_weakadmlocNewton}, and Nguyen-Viehmann \cite{NV_HNstrata}.

In order to state our main result, let us introduce some notations and terminologies. Let $\algclosedperfdfield$ be an algebraically closed perfectoid field of characteristic $p>0$. Denote by $\schff = \schff_\algclosedperfdfield$ the Fargues-Fontaine curve associated to $\algclosedperfdfield$. The Picard group of $\schff$ turns out to be naturally isomorphic to $\Z$, and consequently yields a good Harder-Narasimhan formalism for vector bundles on $\schff$. By a result of Fargues-Fontaine \cite{FF_curve} (and also Kedlaya \cite{Kedlaya_slopefiltrations}), every vector bundle $\Vcal$ on $\schff$ is uniquely determined up to isomorphism by its Harder-Narasimhan polygon $\HN(\Vcal)$.

We can now state our main result as follows:
\begin{thm}\label{existence of exact sequence, intro}
Let $\Dcal, \Ecal$, and $\Fcal$ be vector bundles on $\schff$ such that $\Ecal$ is semistable. There exists a short exact sequence of vector bundles on $\schff$
\[ 0 \longrightarrow \Dcal \longrightarrow \Ecal \longrightarrow \Fcal \longrightarrow 0\]
if and only if the following conditions are satisfied:
\begin{enumerate}[label=(\roman*)]
\item All slopes in $\HN(\Dcal)$ are less than or equal to the slope of $\HN(\Ecal)$. 
\smallskip

\item All slopes in $\HN(\Fcal)$ are greater than or equal to the slope of $\HN(\Ecal)$. 
\smallskip

\item $\HN(\Dcal \oplus \Fcal)$ lies above $\HN(\Ecal)$ with the same endpoints. 
\end{enumerate}
\end{thm}

In the sequel paper \cite{Hong_extvbfinal}, we extend Theorem \ref{existence of exact sequence, intro} to the case where $\Ecal$ is not necessarily semistable. Nonetheless, it is our opinion that Theorem \ref{existence of exact sequence, intro} is worthwhile as an independent statement. In fact, the main result of the article \cite{Hong_extvbfinal} involves a somewhat complicated combinatorial condition and is not at all obviously equivalent to Theorem \ref{existence of exact sequence, intro} in the case where $\Ecal$ is semistable.

\subsection{Outline of the proof}\label{proof outline}$ $

Let us briefly explain our proof of Theorem \ref{existence of exact sequence, intro}. The necessity part of Theorem \ref{existence of exact sequence, intro} is a standard consequence of the slope formalism. Hence the main part of our proof is to establish the sufficiency part of Theorem \ref{existence of exact sequence, intro}. 
When either $\Dcal$ or $\Fcal$ is semistable, we consider the moduli space $\Extdiamond(\Fcal, \Dcal)_\Ecal$ whose $\algclosedperfdfield$-points parametrize exact sequences $0 \to \Dcal \to \Ecal \to \Fcal \to 0$ of vector bundles on $\schff$, and establish its nonemptiness by a dimension analysis. For the general case, we proceed by induction on the number of distinct slopes in $\HN(\Dcal)$ and $\HN(\Fcal)$ using a combinatorial argument that utilizes concavity of HN polygons. 

In order to study of the space $\Extdiamond(\Fcal, \Dcal)_\Ecal$, we adapt the strategies developed in the previous paper \cite{Arizona_extvb}. We make sense of this space as a \emph{diamond} in the sense of Scholze \cite{Scholze_diamonds}, and establish a simple dimension formula for this space when $\Fcal$ is semistable. For our dimension analysis, we prove some combinatorial lemmas involving the HN polygons of vector bundles on $\schff$. 

We remark that the previous version of this paper had a mistake and falsely claimed a similar statement of Theorem \ref{existence of exact sequence, intro} in the case where either $\Dcal$ or $\Fcal$ is semistable. The mistake was to assert that the space $\Extdiamond(\Fcal, \Dcal)_\Ecal$ admits an explicit dimension formula for any vector bundles $\Dcal$, $\Ecal$, and $\Fcal$. However, it turns out that such a dimension formula exists only when $\Fcal$ is semistable, as stated in Proposition \ref{Ext(F, D)_E dimension formula}. 


\subsection*{Acknowledgments} The author would like to sincerely thank David Hansen, Miaofen Chen and the anonymous referee for pointing out a mistake in the previous version of this paper. 

\section{Preliminaries}\label{background}

\subsection{The Fargues-Fontaine curve}$ $

Throughout this paper, we fix an algebraically closed perfectoid field $\algclosedperfdfield$ of characteristic $p >0$. We denote by $\integerring{\algclosedperfdfield}$ the ring of integers of $\algclosedperfdfield$ and by $\witt(\integerring{\algclosedperfdfield})$ the ring of Witt vectors over $\integerring{\algclosedperfdfield}$. We choose a pseudouniformizer $\pseudounif$ of $\algclosedperfdfield$ and write $[\pseudounif]$ for the Teichm$\ddot{\text{u}}$ller lift of $\pseudounif$. The Frobenius map on $\witt(\integerring{\algclosedperfdfield})$ induces a properly discontinuous automorphism $\phi$ on the adic space 
\[
\Ycal:=\Spa(\witt(\integerring{\algclosedperfdfield}))\setminus\{|p[\pseudounif]|=0\}
\]
defined over $\Spa(\Q_p)$. 
\begin{defn} We define the \emph{adic Fargues-Fontaine curve} (associated to $\algclosedperfdfield$) by
\[
\adicff:=\Ycal/\phi^\Z,
\]
and the \emph{schematic Fargues-Fontaine curve} by 
\[\schff:= \Proj \left( \bigoplus_{n \geq 0 } H^0(\Ycal, \trivbundle_\Ycal)^{\phi = p^n} \right).\]
\end{defn}

\begin{remark}
More generally, for any finite extension $\finext$ of $\Q_p$ with ring of integers $\integerring{\finext}$, we can define the Fargues-Fontaine curve as an adic space or a scheme over $\finext$ by replacing $\witt(\integerring{\algclosedperfdfield})$ in the above construction with $\witt_{\integerring{\finext}}(\integerring{\algclosedperfdfield})$, the ring of ramified Witt vectors over $\integerring{\algclosedperfdfield}$ with coefficients in $\integerring{\finext}$. There is also an analogous construction of the equal characteristic Fargues-Fontaine curve as an adic space or a scheme over a finite extension of $\F_p((t))$. Our main result equally holds in these settings with identical proofs. 
\end{remark}

The two incarnations of the Fargues-Fontaine curve are essentially equivalent to us because of the following GAGA type result:
\begin{thm}[{\cite[Theorems 6.3.12 and 8.7.7]{KL_relpadic1}}]\label{GAGA for FF curve}
There exists a natural map of locally ringed spaces 
\[
\adicff \rightarrow \schff
\]
which induces by pullback an equivalence of the categories of vector bundles. 
\end{thm}

In light of Theorem \ref{GAGA for FF curve}, we will identify vector bundles on $\adicff$ with vector bundles on $\schff$. 
Vector bundles on the Fargues-Fontaine curve turn out to behave pleasantly well, essentially due to
the following fact:

\begin{prop}[{\cite[Th\'eor\'eme 5.2.7]{FF_curve}}]\label{FF curve is a curve} The scheme $\schff$ is noetherian and regular of Krull dimension 1 over $\Q_p$. Moreover, it is complete in the sense that every principal divisor on $\schff$ has degree $0$. 
\end{prop}

In particular, the degree map is well-defined on the Picard group of $\schff$, thereby allowing us to define the notion of slope for vector bundles on $\schff$ as follows: 
\begin{defn}
Let $\Vcal$ be a nonzero vector bundle on $\schff$. 
\begin{enumerate}[label=(\arabic*)]
\item We write $\rk(\Vcal)$ for the rank of $\Vcal$ and $\Vcal^\vee$ for the dual of $\Vcal$. 
\smallskip

\item We define the \emph{degree} and \emph{slope} of $\Vcal$ respectively by
\[\deg (\Vcal) := \deg (\wedge^{\rk(\Vcal)} \Vcal) \quad\quad \text{ and } \quad\quad \mu(\Vcal) := \dfrac{\deg(\Vcal)}{\rk(\Vcal)}.\]
\end{enumerate}
\end{defn}

Let $\genresfield$ be the residue field of $\algclosedperfdfield$, and let $\fracwitt$ be the fraction field of the ring of Witt vectors over $\genresfield$. Recall that an \emph{isocrystal} over $\genresfield$ is a finite dimensional vector space over $\fracwitt$ with a Frobenius semi-linear automorphism. 

\begin{lemma}\label{functor from isocrystals to bundles on FF curve}
There exists a functor from the category of isocrystals over $\genresfield$ to the category of vector bundles on $\schff$ which is compatible with direct sums, 
duals, ranks, degrees, and slopes. 
\end{lemma}

\begin{proof}
Let us write
\[ B := H^0(\Ycal, \trivbundle_\Ycal) \quad\quad \text{ and } \quad\quad P:= \bigoplus_{n \geq 0} B^{\phi = \varpi^n}.\]
The desired functor is given by associating to each isocrystal $N$ over $\genresfield$ the vector bundle $\Ecal(N)$ on $\schff$ which corresponds to the graded $P$-module
\[\bigoplus_{n \geq 0 } \left( N^\vee \otimes_{\fracwitt} B \right)^{\phi = \varpi^n},\]
where $N^\vee$ denotes the dual isocrystal of $N$. 
\end{proof}

\begin{defn}\label{o-r-over-s}
Given $\lambda \in \Q$, we write $\trivbundle(\lambda)$ for the vector bundle on $\schff$ that corresponds to the unique simple isocrystal over $\genresfield$ of slope $\lambda$ under the functor in Lemma \ref{functor from isocrystals to bundles on FF curve}. 
\end{defn}

\begin{lemma}\label{basic properties of stable bundles}
Let $d$ and $r$ be relatively prime integers with $r>0$. 
\begin{enumerate}[label=(\arabic*)]
\item\label{rank and degree of stable bundles} The bundle $\trivbundle(d/r)$ has rank $r$, degree $d$, and slope $d/r$.
\smallskip

\item\label{tensor product of stable bundles} For any relatively prime integers $d'$ and $r'$ with $r'>0$, we have
\[
\trivbundle\left(\frac dr\right)\otimes\trivbundle\left(\frac{d'}{r'}\right)\simeq\trivbundle\left(\frac dr+\frac{d'}{r'}\right)^{\dsm\gcd(r r',dr'+d'r)}.
\]
In particular, the bundle $\trivbundle(d/r)\otimes\trivbundle(d'/r')$ has rank $rr'$, degree $dr'+d'r$, and slope $d/r+d'/r'$.
\smallskip

\item\label{dual of stable bundles} $\trivbundle(d/r)^\vee \simeq \trivbundle(-d/r)$. 
\end{enumerate}
\end{lemma}
\begin{proof}
By Lemma \ref{functor from isocrystals to bundles on FF curve}, all statements follow immediately from the corresponding statements for isocrystals over $\genresfield$. 
\end{proof}

\begin{prop}[{\cite[Proposition 5.6.23]{FF_curve}, \cite[Proposition 4.1.3]{Kedlaya_slopefiltrations}}]\label{cohomologies of stable bundles}
For every $\lambda \in \Q$, we have the following statements:
\begin{enumerate}[label=(\arabic*)]
\item $H^0(\schff, \trivbundle(\lambda)) = 0$ if and only if $\lambda<0$. 

\item $H^1(\schff, \trivbundle(\lambda)) = 0$ if and only if $\lambda \geq 0$. 
\end{enumerate}
\end{prop}



\begin{defn}
A vector bundle $\Vcal$ on $\schff$ is \emph{semistable} if 
every subbundle $\Wcal$ of $\Vcal$ satisfies the inequality $\mu(\Wcal) \leq \mu(\Vcal)$. 
\end{defn}


We can now state the classification theorem for vector bundles on $\schff$ as follows:

\begin{theorem}[{\cite[Th\'eor\'eme 8.2.10]{FF_curve}}]\label{existence of HN decomp}
Let $\Vcal$ be a nonzero vector bundle on $\schff$. 
\begin{enumerate}[label=(\arabic*)]
\item $\Vcal$ is semistable of slope $\lambda$ if and only if it is isomorphic to $\trivbundle(\lambda)^{\oplus m}$ for some $m$. 
\smallskip

\item $\Vcal$ admits a direct sum decomposition 
\begin{equation}\label{HN decomposition}
\Vcal \simeq \bigoplus_{i=1}^l \trivbundle(\lambda_i)^{\oplus m_i}
\end{equation}
where the $\lambda_i$'s are rational numbers 
with $\lambda_1 > \lambda_2 > \cdots > \lambda_l$. 
\end{enumerate}
\end{theorem}

\begin{defn}\label{def of HN decomposition and HN polygon}
Let $\Vcal$ be a nonzero vector bundle on $\schff$. 
\begin{enumerate}[label = (\arabic*)]
\item We refer to the decomposition \eqref{HN decomposition} 
in Theorem \ref{existence of HN decomp} 
as the \emph{Harder-Narasimhan (HN) decomposition} of $\Vcal$. 
\smallskip

\item We refer to the numbers $\lambda_i$ in the HN decomposition as the \emph{Harder-Narasimhan (HN) slopes} of $\Vcal$, or often simply as the \emph{slopes} of $\Vcal$.
\smallskip

\item We write $\mumax(\Vcal)$ (resp. $\mumin(\Vcal)$) for the maximum (resp. minimum) HN slope of $\Vcal$; in other words, we set $\mumax(\Vcal) := \lambda_1 \text{ and } \mumin(\Vcal) := \lambda_l$. 
\smallskip

\item For every $\mu \in \Q$, we define the direct summands
\[\Vcal^{\geq \mu} := \bigoplus_{\lambda_i \geq \mu}\trivbundle(\lambda_i)^{\oplus m_i} \quad\quad \text{ and } \quad\quad\Vcal^{\leq \mu} := \bigoplus_{\lambda_i \leq \mu}\trivbundle(\lambda_i)^{\oplus m_i},\]
and similarly define $\Vcal^{>\mu}$ and $\Vcal^{<\mu}$. 
\smallskip

\item We define the \emph{Harder-Narasimhan (HN) polygon} of $\Vcal$, denoted by $\HN(\Vcal)$, as the upper convex hull of the points $(0, 0)$ and $\big(\rk(\Vcal^{\geq \lambda_i}), \deg(\Vcal^{\geq \lambda_i})\big)$.
\smallskip

\item Given a convex polygon $P$ adjoining $(0, 0)$ and $(\rk(\Vcal), \deg(\Vcal))$, we write $\HN(\Vcal) \leq P$ if each point on $\HN(\Vcal)$ lies on or below $P$. 
\end{enumerate}
\end{defn}

\begin{cor}
For an arbitrary vector bundle $\Vcal$ on $\schff$, its isomorphism class is determined by the HN polygon $\HN(\Vcal)$, with the slopes of $\Vcal$ precisely being the slopes in $\HN(\Vcal)$.  
\end{cor}

We conclude this subsection by extending the construction of the Fargues-Fontaine curve to relative settings. Let $S = \Spa (\genperfdring, \genperfdring^+)$ be an affinoid perfectoid space over $\Spa (\algclosedperfdfield)$, and let $\pseudounif_\genperfdring$ be a pseudouniformizer  of $\genperfdring$. We write $\witt(\genperfdring^+)$ for the ring of Witt vectors over $\genperfdring^+$ and $[\pseudounif_\genperfdring]$ for the Teichm$\ddot{\text{u}}$ller lift of $\pseudounif_\genperfdring$. As in the absolute setting, the Frobenius map on $\witt(\genperfdring^+)$ induces a properly discontinuous automorphism $\phi$ on the adic space
\[
\Ycal_S:=\Spa(\witt(\genperfdring^+))\setminus\{|p[\pseudounif_\genperfdring]|=0\}
\]
defined over $\Spa(\Q_p)$. 

\begin{defn}
Given an affinoid perfectoid space $S = \Spa (\genperfdring, \genperfdring^+)$ over $\Spa (\algclosedperfdfield)$, we define the \emph{adic Fargues-Fontaine curve} associated to $S$ by
\[
\adicff_S:=\Ycal_S/\phi^\Z,
\]
and the \emph{schematic Fargues-Fontaine curve} associated to $S$ by 
\[\schff_S:= \Proj \left( \bigoplus_{n \geq 0 } H^0(\Ycal_S, \trivbundle_{\Ycal_S})^{\phi = p^n} \right).\]
For an arbitrary perfectoid space $S$ over $\Spa(\algclosedperfdfield)$ with an affinoid cover $S = \bigcup S_i$, we define the adic Fargues-Fontaine curve $\adicff_S$ by gluing the $\adicff_{S_i}$. 
\end{defn}

\begin{remark}
The schematic Fargues-Fontaine curve $\schff_S$ is defined only for affinoid perfectoid spaces $S$; in fact, for an arbitrary perfectoid space $S$ over $\Spa(\algclosedperfdfield)$ with an affinoid cover $S = \bigcup S_i$, the schematic curves $\schff_{S_i}$ do not glue in general. In addition, the readers should be aware that the relative Fargues-Fontaine curve $\adicff_S$ is not related to $\adicff$ by a base change, as neither $\adicff$ nor $\adicff_S$ is defined over $\Spa(\algclosedperfdfield)$. 
\end{remark}


\subsection{Diamonds}\label{diamonds}$ $

In this subsection, we collect some basic facts about diamonds following \cite{Scholze_diamonds}. 

\begin{defn}
Let $\Perfd$ denote the category of perfectoid spaces in characteristic $p$. 

\begin{enumerate}[label=(\arabic*)]
\item A morphism $Y \to Z$ of affinoid perfectoid spaces is \emph{affinoid pro-\'etale} if it can be written as a cofiltered limit of \'etale morphisms $Y_i \to Z$ of affinoid perfectoid spaces. 
\smallskip

\item A morphism $f: Y \to Z$ of perfectoid spaces is \emph{pro-\'etale} if there exist open affinoid covers $Z = \bigcup U_i$ and $Y = \bigcup V_{i, j}$ such that $f|_{V_{i, j}}$ factors through an affinoid pro-\'etale morphism $V_{i, j} \to U_i$. 
\smallskip

\item A pro-\'etale morphism $f: Y \to Z$ of perfectoid spaces is called a \emph{pro-\'etale cover} if for any quasicompact open subset $U \subset Z$, there exists some quasicompact open subset $V \subset Y$ with $f(V) = U$. 
\smallskip

\item The \emph{big pro-\'etale site} is the site on $\Perfd$ with covers given by pro-\'etale covers. 
\smallskip

\item A sheaf $Y$ for the big pro-\'etale site on $\Perfd$ is called a \emph{diamond} if $Y$ can be written as a quotient $Z/R$, where $Z$ is representable by a perfectoid space with a pro-\'etale equivalence relation $R$ on $Z$.  
\smallskip

\item For a diamond $Y \simeq Z/R$ with a perfectoid space $Z$ and a pro-\'etale equivalence relation $R$, we define its topological space by $|Y|:= |Z|/|R|$, where $|Z|$ and $|R|$ respectively denote the topological spaces for $Z$ and $R$. 
\smallskip

\item For a diamond $Y$, we define its \emph{dimension} to be the Krull dimension of $|Y|$. 
\end{enumerate}
\end{defn}

\begin{remark}
For a diamond $Y$, its topological space $|Y|$ does not depend on the choice of presentation $Y \simeq Z/R$ as the quotient of a perfectoid space
$Z$ by a pro-\'etale equivalence relation $R$.
\end{remark}

\begin{prop}[{\cite[Corollary 8.6]{Scholze_diamonds}}]\label{yoneda for perfectoid spaces}
The big pro-\'etale site is subcanonical. In other words, for every $Z \in \Perfd$ the functor $\mathrm{Hom}(-, Z)$ is a sheaf for the big pro-\'etale site. 
\end{prop}

\begin{remark}
By Proposition \ref{yoneda for perfectoid spaces}, we will often identify a perfectoid space $Z$ in characteristic $p$ with the functor $\mathrm{Hom}(-, Z)$ on $\Perfd$. 
\end{remark}


\begin{defn}
Let $Y$ be a diamond. 

\begin{enumerate}[label=(\arabic*)]
\item We say that $Y$ is \emph{quasicompact} if it admits a presentation $Y \simeq Z/R$ for some quasicompact perfectoid space $Z$ and a pro-\'etale equivalence relation $R$ on $Z$.
\smallskip

\item We say that $Y$ is \emph{quasiseparated} if $U \times_Y V$ is quasicompact for any morphisms $U \to Y$ and $V \to Y$ of diamonds with $U, V$ quasicompact. 
\smallskip

\item We say that $Y$ is \emph{partially proper} if it is quasiseparated with the property that for all characteristic $p$ affinoid perfectoid pair $(\genperfdring, \genperfdring^+)$ the restriction map
\[ Y(\genperfdring, \genperfdring^+) \to Y(\genperfdring, \integerring{\genperfdring})\]
is bijective where $\integerring{\genperfdring}$ denotes the ring of power-bounded elements in $\genperfdring$. 
\smallskip

\item We say that $Y$ is \emph{spatial} if it is quasicompact and quasiseparated with a neighborhood basis of $|Y|$ given by $\Set{|U|: U \subset Y \text{ quasicompact open subdiamonds}}$.  
\smallskip

\item We say that $Y$ is \emph{locally spatial} if it admits a covering by spatial open subdiamonds. 
\end{enumerate}
\end{defn}

\begin{remark}
In general, quasicompactness (resp. quasiseparatedness) of a diamond $Y$ is not equivalent to quasicompactness (resp. quasiseparatedness) of its topological space $|Y|$. 
\end{remark}

\begin{prop}[{\cite[Lemma 3.2.3 and Lemma 3.3.4]{Arizona_extvb}}]\label{quotient diamond by an action of a profinite group}
Let $Y$ be a spatial diamond with a free $\constgrpsch{G}$-action for some profinite group $G$. Then $Y/\constgrpsch{G}$ is a spatial diamond with 
\[\dim Y/\constgrpsch{G} = \dim Y.\]
\end{prop}

\section{Semistable vector bundles arising from extensions}
  
\subsection{Moduli spaces of extensions}\label{moduli of extensions}$ $


In this subsection, we define and study diamonds that parametrize maps or extensions between two given vector bundles on $\adicff$. Let us denote by $\Perfd_{/\Spa (\algclosedperfdfield)}$ the category of perfectoid spaces over $\Spa(\algclosedperfdfield)$. By construction, the relative Fargues-Fontaine curve $\adicff_S$ for any $S \in \Perfd_{/\Spa (\algclosedperfdfield)}$ comes with a natural map $\adicff_S \to \adicff$. 

\begin{defn} Let $\Ecal$ and $\Fcal$ be vector bundles on the Fargues-Fontaine curve $\adicff$. For any $S \in \Perfd_{/\Spa(\algclosedperfdfield)}$, we write $\Ecal_S$ and $\Fcal_S$ for the pullbacks of $\Ecal$ and $\Fcal$ along the map $\adicff_S \to \adicff$.
\begin{enumerate}[label=(\arabic*)]
\item $\Hcal^i(\Ecal)$ is the pro-\'etale sheafification of the functor which associates to each $S \in \Perfd_{/\Spa (\algclosedperfdfield)}$ the group $H^i(\adicff_S, \Ecal_S)$. 
\smallskip

\item $\Hom(\mathcal{E},\Fcal)$ is the functor
which associates $S \in \Perfd_{/\Spa(\algclosedperfdfield)}$ to the set of $\trivbundle_{\adicff_S}$-module
maps $\Ecal_S\to\Fcal_S$. 
\smallskip

\item $\Surj(\Ecal, \Fcal)$ is the functor which associates $S \in \Perfd_{/\Spa(\algclosedperfdfield)}$ to the set of surjective $\trivbundle_{\adicff_S}$-module maps $\Ecal_S \surj \Fcal_S$. 
\smallskip

\item $\Surj(\Ecal, \Fcal)_\Dcal$ is the functor which associates $S \in \Perfd_{/\Spa(\algclosedperfdfield)}$ to the set of surjective $\trivbundle_{\adicff_S}$-module maps $\Ecal_S \surj \Fcal_S$ whose kernel becomes isomorphic to $\Dcal$ after pulling back along the map $\adicff_{\finext, \overline{x}} \to\adicff_{\finext, S}$ for any geometric point $\overline{x}$. 
\smallskip

\item $\Inj(\Ecal, \Fcal)$ is the functor which associates $S \in \Perfd_{/\Spa(\algclosedperfdfield)}$ to the set of 
$\trivbundle_{\adicff_S}$-module
maps $\Ecal_{S}\to\Fcal_{S}$ whose pullback along the map $\adicff_{\overline{x}} \to\adicff_S$ for any geometric point $\overline{x}
\to S$
gives an injective $\trivbundle_{\adicff_{\overline{x}}}$-module map.
\smallskip


\item $\Aut(\Ecal)$ is the functor which associates $S \in \Perfd_{/\Spa(\algclosedperfdfield)}$ to the group of $\trivbundle_{\adicff_S}$-module automorphisms of $\Ecal_S$.
\smallskip

\item $\Extdiamond(\Fcal, \Dcal)$ is the functor which associates to each $S \in \Perfd_{/\Spa (\algclosedperfdfield)}$ the set of isomorphism classes of extensions of $\Fcal_S$ by $\Dcal_S$. 
\smallskip

\item $\Extdiamond(\Fcal, \Dcal)_\Ecal$ is the functor which associates to each $S \in \Perfd_{/\Spa (\algclosedperfdfield)}$ the set of isomorphism classes of extensions of $\Fcal_S$ by $\Dcal_S$ whose pullback along the map $\adicff_{\finext, \overline{x}} \to\adicff_{\finext, S}$ for any geometric point $\overline{x}
\to S$ yields a short exact sequence 
\[ 0 \longrightarrow \Dcal \longrightarrow \Ecal \longrightarrow \Fcal \longrightarrow 0.\]
\end{enumerate}
\end{defn}

\begin{remark}
We have canonical identifications 
\[\Hom(\Ecal, \Fcal) \cong \Hcal^0(\Ecal^\vee \otimes \Fcal)\quad\quad \text{ and } \quad\quad \Extdiamond(\Fcal, \Dcal) \cong \Hcal^1(\Fcal^\vee \otimes \Dcal).\] 
\end{remark}




\begin{prop}[{\cite[Propositions 3.3.2, 3.3.5, 3.3.6, 3.3.7, and 3.3.13]{Arizona_extvb}}]\label{moduli of bundle maps fund facts}
Let $\Ecal$ and $\Fcal$ be vector bundles on $\adicff$. 
\begin{enumerate}[label=(\arabic*)]
\item If $\Ecal$ is semistable of slope $0$, then there is a natural identification $\Hcal^0(\Ecal) \cong  \constgrpsch{\Q_p}^{\rk(\Ecal)}$.
\smallskip

\item $\Hom(\Ecal, \Fcal)$ is a partially proper and locally spatial diamond over $\algclosedperfdfield$, equidimensional of dimension $\deg(\Ecal^\vee \otimes \Fcal)^\nonneg$. 
\smallskip

\item Every nonempty open subdiamond of $\Hom(\Ecal, \Fcal)$ has an $\algclosedperfdfield$-point. 
\smallskip

\item $\Surj(\Ecal, \Fcal)$ and $\Inj(\Ecal, \Fcal)$ are both open, partially proper and locally spatial subdiamonds of $\Hom(\Ecal, \Fcal)$. 
\smallskip

\item $\Aut(\Ecal)$ is a partially proper and locally spatial diamond over $\algclosedperfdfield$, equidimensional of dimension $\deg(\Ecal^\vee \otimes \Ecal)^\nonneg$. 
\smallskip

\item $\Surj(\Ecal, \Fcal)_\Dcal$ is a partially proper and locally spatial diamond over $\algclosedperfdfield$. 
\end{enumerate}
\end{prop}

\begin{remark}
The work of Le Bras \cite{LeBras_BCspaces} shows that the diamonds $\Hcal^i(\Ecal)$ and $\Hom(\Ecal, \Fcal)$ also have the structure of Banach-Colmez spaces as defined by Colmez \cite{Colmez_BCspace}. 
\end{remark}

\begin{prop}[{\cite[Theorem 1.1.2]{Arizona_extvb}}]\label{extension theorem with semistable kernel and cokernel}
Let $\Dcal$, $\Ecal$, and $\Fcal$ be vector bundles on $\adicff$ such that $\Dcal$ and $\Fcal$ are semistable with $\mu(\Dcal) < \mu(\Fcal)$. 
There exists a short exact sequence
\[ 0 \longrightarrow \Dcal \longrightarrow \Ecal \longrightarrow \Fcal \longrightarrow 0\]
if and only if we have $\HN(\Ecal) \leq \HN(\Dcal \oplus \Fcal)$. 
\end{prop}

\begin{prop}\label{H1 diamond is locally spatial}
Let $\Ecal$ be a vector bundle on $\adicff$ with $\mumax(\Ecal)<0$. 
\begin{enumerate}[label=(\arabic*)]
\item $\Hcal^1(\Ecal)$ is a partially proper and locally spatial diamond over $\algclosedperfdfield$, equidimensional of dimension $\deg(\Ecal^\vee)^\nonneg$. 
\smallskip

\item Every nonempty open subdiamond of $\Hcal^1(\Ecal)$ has an $\algclosedperfdfield$-point. 
\end{enumerate}
\end{prop}

\begin{proof}
Let us write the HN decomposition of $\Ecal$ as
\[ \Ecal \simeq \bigoplus_{i=1}^l \trivbundle(\lambda_i)^{\oplus m_i}\]
with $\lambda_i<0$ for each $i = 1, \cdots, l$. 
We also set
\[ r_i:= \rk\left(\trivbundle(\lambda_i)\right) \quad\quad \text{ and } \quad\quad d_i := \deg\left(\trivbundle(\lambda_i)\right).\]
By Proposition \ref{extension theorem with semistable kernel and cokernel}, each $\trivbundle(\lambda_i)^{\oplus m_i}$ fits into a short exact sequence
\[ 0 \longrightarrow \trivbundle(\lambda_i)^{\oplus m_i} \longrightarrow \trivbundle^{\oplus m_i(r_i - d_i)} \longrightarrow \trivbundle(1)^{\oplus - m_i d_i} \longrightarrow 0.\]
We take the direct sum of all such exact sequences to obtain a short exact sequence
\begin{equation}\label{H1 diamond key short exact sequence}
0 \longrightarrow \Ecal \longrightarrow \trivbundle^{\oplus (r-d)} \longrightarrow \trivbundle(1)^{\oplus -d} \longrightarrow 0
 \end{equation}
with $r = \rk(\Ecal)$ and $d = \deg(\Ecal)$, and consequently get a long exact sequence
\[0 \longrightarrow \Hcal^0(\Ecal) \longrightarrow \Hcal^0(\trivbundle^{\oplus (r-d)}) \longrightarrow \Hcal^0(\trivbundle(1)^{\oplus -d}) \longrightarrow \Hcal^1(\Ecal) \longrightarrow \Hcal^1(\trivbundle^{\oplus (r-d)}).\]
Moreover, by Proposition \ref{cohomologies of stable bundles} and Proposition \ref{moduli of bundle maps fund facts} we have
\[ \Hcal^0(\Ecal) = 0, \quad\quad \Hcal^0(\trivbundle^{\oplus (r-d)}) \cong \constgrpsch{\Q_p}^{r-d}, \quad\quad \Hcal^1(\trivbundle^{\oplus (r-d)}) = 0.\]
We thus find a presentation
\[ \Hcal^1(\Ecal) \simeq \Hcal^0(\trivbundle(1)^{\oplus -d})/ \constgrpsch{\Q_p}^{r-d} \simeq \Hom(\trivbundle, \trivbundle(1)^{\oplus -d})/ \constgrpsch{\Q_p}^{r-d},\]
thereby deducing the desired statements by Proposition \ref{quotient diamond by an action of a profinite group} and Proposition \ref{moduli of bundle maps fund facts}. 
\end{proof}

\begin{remark}
The above argument is largely inspired by the proof of \cite[Proposition 3.3.2]{Arizona_extvb}. It is also presented by Hansen at the Montreal workshop for the geometrization of the local Langlands program held in 2019. 

It is worthwhile to note that our use of Proposition \ref{extension theorem with semistable kernel and cokernel} is not essential and is only for brevity. For example, we can prove Proposition \ref{H1 diamond is locally spatial} based only on some elementary properties of the Fargues-Fontaine curve, as in the work of Fargues-Scholze \cite[Proposition II.2.5]{FS_geomLL}. In fact, Fargues-Scholze \cite[Theorem II.2.14]{FS_geomLL} uses a special case of this result in an essential way to give a new conceptual proof of Theorem \ref{existence of HN decomp}. 
\end{remark}




\begin{lemma}[{\cite[Lemma 3.3.14]{Arizona_extvb}}]\label{Sur(E, F)_D dimension formula}
Let $\Dcal, \Ecal$, and $\Fcal$ be vector bundles on $\adicff$ such that $\Fcal$ is semistable. Then $\Surj(\Ecal, \Fcal)_\Dcal$ is either empty or equidimensional with
\[ \dim \Surj(\Ecal, \Fcal)_\Dcal = \deg(\Dcal^\vee \otimes \Ecal)^\nonneg - \deg(\Dcal^\vee \otimes \Dcal)^\nonneg.\]
\end{lemma}

\begin{remark}
The diamond $\Surj(\Ecal, \Fcal)_\Dcal$ is quite obscure if $\Fcal$ is not semistable. For instance, we are highly doubtful that $\Surj(\Ecal, \Fcal)_\Dcal$ admits an explicit dimension formula when $\Fcal$ is not semistable. 
\end{remark}

\begin{prop}\label{Ext(F, D)_E dimension formula}
Let $\Dcal, \Ecal$, and $\Fcal$ be vector bundles on $\adicff$ with $\mumax(\Dcal) < \mumin(\Fcal)$. 
%
\begin{enumerate}[label=(\arabic*)]
\item $\Extdiamond(\Fcal, \Dcal)_\Ecal$ is a partially proper and locally spatial diamond over $\algclosedperfdfield$. 
\smallskip

\item If $\Ecal$ is semistable, then $\Extdiamond(\Fcal, \Dcal)_\Ecal$ is an open subdiamond of $\Extdiamond(\Fcal, \Dcal)$. 
\smallskip

\item If $\Fcal$ is semistable, then $\Extdiamond(\Fcal, \Dcal)_\Ecal$ is either empty or equidimensional with 
\[\dim \Extdiamond(\Fcal, \Dcal)_\Ecal = \deg(\Dcal^\vee \otimes \Ecal)^\nonneg  - \deg(\Ecal^\vee \otimes \Ecal)^\nonneg.\]
\end{enumerate}
\end{prop}

\begin{proof}
By Proposition \ref{basic properties of stable bundles}, we find that all slopes of $\Fcal^\vee \otimes \Dcal$ are negative. Hence Proposition \ref{H1 diamond is locally spatial} implies that $\Extdiamond(\Fcal, \Dcal) \cong \Hcal^1(\Fcal^\vee \otimes \Dcal)$ is a locally spatial diamond over $\algclosedperfdfield$. Let us choose a presentation $\Extdiamond(\Fcal, \Dcal) \simeq T/R$ for some perfectoid space $T$ and a pro-\'etale equivalence relation $R$. Let $\Vcal$ be the vector bundle on $\schff_T$ which fits into the ``universal" exact sequence
\[0 \longrightarrow \Dcal_T \longrightarrow \Vcal \longrightarrow \Fcal_T \longrightarrow 0.\]
We define 
\begin{align*}
|T|_{\leq \HN(\Ecal)}&:= \Set{ x \in |T|: \HN(\Vcal_x) \leq \HN(\Ecal)}, \\
|T|_{\geq \HN(\Ecal)}&:= \Set{ x \in |T|: \HN(\Vcal_x) \geq \HN(\Ecal)}.
\end{align*}
By the result of Kedlaya-Liu \cite[Theorem 7.4.5]{KL_relpadic1}, 
the set $|T|_{\leq \HN(\Ecal)}$ (resp. $|T|_{\geq \HN(\Ecal)}$) is open (resp. closed) in $|T|$. Moreover, both $|T|_{\leq \HN(\Ecal)}$ and $|T|_{\geq \HN(\Ecal)}$ are stable under generalizations. 
Hence the image of $|T|_{\leq \HN(\Ecal)} \cap |T|_{\geq \HN(\Ecal)}$ under the quotient map $|T| \to |\Extdiamond(\Fcal, \Dcal)|$ is a locally closed and generalizing subset $|\Extdiamond(\Fcal, \Dcal)|_{\HN(\Ecal)}$ of $|\Extdiamond(\Fcal, \Dcal)|$. Adapting the argument of Scholze \cite[Proposition 11.20]{Scholze_diamonds},
we find that $|\Extdiamond(\Fcal, \Dcal)|_{\HN(\Ecal)}$ gives rise to a locally spatial subdiamond $\Extdiamond(\Fcal, \Dcal)_{\HN(\Ecal)}$ of $\Extdiamond(\Fcal, \Dcal)$ with an identification
\begin{equation*}\label{Ext(F, D)_E identification as a HN stratum}
\Extdiamond(\Fcal, \Dcal)_{\HN(\Ecal)} \cong \Extdiamond(\Fcal, \Dcal)_\Ecal
\end{equation*}
as a functor on $\Perfd_{/\Spa (\algclosedperfdfield)}$. Therefore we deduce that $\Extdiamond(\Fcal, \Dcal)_\Ecal$ is a locally spatial diamond over $\algclosedperfdfield$, and also obtain its partial properness from the result of Kedlaya-Liu \cite[Theorem 8.7.7]{KL_relpadic1}.
Moreover, if $\Ecal$ is semistable, then we have $|T|_{\geq \HN(\Ecal)} = |T|$ and consequently find that $\Extdiamond(\Fcal, \Dcal)_\Ecal \cong \Extdiamond(\Fcal, \Dcal)_{\HN(\Ecal)}$ is an open subdiamond of $\Extdiamond(\Fcal, \Dcal)$. 

For the last  statement, let us now assume that $\Fcal$ is semistable. Let $\widetilde{\Extdiamond}(\Fcal, \Dcal)_\Ecal$ be the functor which associates each $S \in \Perfd_{/\Spa (\algclosedperfdfield)}$ to the set of isomorphism classes of short exact sequences
\[ 0 \longrightarrow \Dcal_S \longrightarrow \Ecal_S \longrightarrow \Fcal_S \longrightarrow 0.\]
We may identify $\widetilde{\Extdiamond}(\Fcal, \Dcal)_\Ecal$ as an $\Aut(\Ecal)$-torsor over $\Extdiamond(\Fcal, \Dcal)_\Ecal$ by rigidifying an extension of $\Fcal_S$ by $\Dcal_S$ for each $S \in \Perfd_{/\Spa (\algclosedperfdfield)}$. Similarly, we may identify $\widetilde{\Extdiamond}(\Fcal, \Dcal)_\Ecal$ as an $\Aut(\Dcal)$-torsor over $\Surj(\Ecal, \Fcal)_\Dcal$ by rigidifying a surjective map $\Ecal_S \surj \Fcal_S$ (and its kernel) for each $S \in \Perfd_{/\Spa (\algclosedperfdfield)}$. Therefore, if $\Extdiamond(\Fcal, \Dcal)_\Ecal$ is not empty, we find
\begin{align*}
\dim \Extdiamond(\Fcal, \Dcal)_\Ecal &= \dim \widetilde{\Extdiamond}(\Fcal, \Dcal)_\Ecal - \dim \Aut(\Ecal)\\
&= \dim \Surj(\Ecal, \Fcal)_\Dcal + \dim \Aut(\Dcal) - \dim \Aut(\Ecal) \\
&= \deg(\Dcal^\vee \otimes \Ecal)^\nonneg  - \deg(\Ecal^\vee \otimes \Ecal)^\nonneg
\end{align*}
by Proposition \ref{quotient diamond by an action of a profinite group}, Proposition \ref{moduli of bundle maps fund facts}, and Lemma \ref{Sur(E, F)_D dimension formula}. 
\end{proof}



\subsection{Main theorem}\label{main theorem}$ $

We now aim to establish our main result classifying all vector bundles on $\schff$ which realize a given semistable vector bundle as their extension. 

\begin{defn}
Let $\Vcal$ be a vector bundle on $\schff$ with HN decomposition
\[ \Vcal \simeq \bigoplus_{i=1}^l \trivbundle(\lambda_i)^{\oplus m_i}\]
where the $\lambda_i$'s are in strictly descending order. We define the \emph{HN vectors} of $\Vcal$ by
\[\HNvec(\Vcal):= (v_i)_{1 \leq i \leq l}\]
where $v_i:= \big(\rk(\trivbundle(\lambda_i)^{\oplus m_i}), \deg(\trivbundle(\lambda_i)^{\oplus m_i})\big)$ is the vector that represents the $i$-th line segment in $\HN(\Vcal)$, and write $\mu(v_i):= \lambda_i$ for the slope of $v_i$. 
\end{defn}
\begin{center}
\begin{figure}[h]
\begin{tikzpicture}[scale=0.7]
		\coordinate (left) at (0, 0);
		\coordinate (q1) at (2.5, 3.5);
		\coordinate (q2) at (4, 4);
		\coordinate (q3) at (5, 3.7);

		\coordinate (p0) at (1, 3);
		\coordinate (p1) at (4, 5);
		\coordinate (p2) at (7, 4.5);
		\coordinate (p3) at (9, 2.5);
				

		\draw[step=1cm,thick] (left) -- node[left] {$v_1$} (p0);
		\draw[step=1cm,thick] (p0) -- node[above] {$v_2$} (p1);
		\draw[step=1cm,thick] (p1) -- node[above] {$v_3$} (p2);
		\draw[step=1cm,thick] (p2) -- node[right] {$v_4$} (p3);

		\draw [fill] (left) circle [radius=0.05];
		
		\draw [fill] (p0) circle [radius=0.05];		
		\draw [fill] (p1) circle [radius=0.05];		
		\draw [fill] (p2) circle [radius=0.05];		
		\draw [fill] (p3) circle [radius=0.05];		


		
		\path (p3) ++(1, 0) node {$\HN(\Vcal)$};


\end{tikzpicture}
\setlength{\belowcaptionskip}{-0.3in}
\caption{Vector representation of $\HN(\Vcal)$.}\label{diamond_Lemma_fig}
\end{figure}
\end{center}

\begin{lemma}[{\cite[Lemma 2.3.4]{Arizona_extvb}}]\label{geometric representation of degrees}
Let $\Ecal$ and $\Fcal$ be vector bundles on $\schff$ with HN vectors $\HNvec(\Ecal) = (e_i)$ and $\HNvec(\Fcal) = (f_j)$. 
Then we have an identity
\[ \deg(\Ecal^\vee \otimes \Fcal)^\nonneg = \sum_{\mu(e_i) \leq \mu(f_j)} e_i \times f_j\]
where $e_i \times f_j$ denotes the two-dimensional cross product of the vectors $e_i$ and $f_j$. 
\end{lemma}

\begin{proof}
The assertion is straightforward to verify using Lemma \ref{basic properties of stable bundles}. 
\end{proof}

\begin{remark}
Recall that the two-dimensional cross product of two vectors $v = (x_1, y_1)$ and $w = (x_2, y_2)$ is defined by $v \times w := x_1 y_2 - x_2 y_1$. 
\end{remark}

\begin{lemma}\label{nonneg degree for dominating HN polygons}
Let $\Ecal$ and $\Fcal$ be vector bundles on $\schff$ such that $\HN(\Ecal)$ lies on or below $\HN(\Fcal)$. For every vector bundle $\Qcal$ on $\schff$, we have
\[ \deg(\Qcal^\vee \otimes \Ecal)^\nonneg \leq \deg(\Qcal^\vee \otimes \Fcal)^\nonneg.\]
\end{lemma}

\begin{proof}
It suffices to consider the case where we have $\Qcal = \trivbundle(\lambda)$ for some $\lambda \in \Q$, as the general case will follow from this special case using the HN decomposition of $\Qcal$. Then we note by Lemma \ref{basic properties of stable bundles} that $\HN(\Qcal^\vee \otimes \Ecal)$ and $\HN(\Qcal^\vee \otimes \Fcal)$ are respectively obtained from $\HN(\Ecal)$ and $\HN(\Fcal)$ via the composition of the following transformations:
\begin{itemize}
\item a shear transformation that makes each slope decrease by $\lambda$, and
\smallskip

\item a dilation by the factor $\rk(\Qcal) = \rk(\trivbundle(\lambda))$.
\end{itemize}
Hence we find that $\HN(\Qcal^\vee \otimes \Ecal)$ lies on or below $\HN(\Qcal^\vee \otimes \Fcal)$, and in turn deduce the desired inequality by observing that $\deg(\Qcal^\vee \otimes \Ecal)^\nonneg$ and $\deg(\Qcal^\vee \otimes \Fcal)^\nonneg$ respectively represent the maximum $y$-coordinates of $\HN(\Qcal^\vee \otimes \Ecal)$ and $\HN(\Qcal^\vee \otimes \Fcal)$. 
\end{proof}

\begin{prop}\label{key inequality for semistable quotient}
Let $\Dcal, \Ecal, \Fcal$, and $\Vcal$ be vector bundles on $\schff$ with the following properties:
\begin{enumerate}[label=(\roman*)]
\item\label{dominant HN polygons for kernel} $\HN(\Dcal)$ lies on or below $\HN(\Vcal)$. 
\smallskip


\item\label{HN polygon inequality for extensions, aux} $\HN(\Ecal) \leq \HN(\Vcal) \leq \HN(\Dcal \oplus \Fcal)$ with $\Ecal$ and $\Fcal$ being semistable.
\smallskip


\item\label{slope dominance between kernel and image, aux} $\mumax(\Dcal) \leq \mu(\Ecal) \leq \mu(\Fcal)$. 
\end{enumerate}
Then we have an inequality
\[\deg(\Dcal^\vee \otimes \Vcal)^\nonneg - \deg(\Vcal^\vee \otimes \Vcal)^\nonneg \leq \deg(\Dcal^\vee \otimes \Fcal)^\nonneg\]
with equality if and only if $\Vcal$ is isomorphic to $\Ecal$. 
\end{prop}

\begin{proof}
The desired inequality can be written as
\begin{equation*}\label{key inequality for step 2 alt form} 
\deg(\Dcal^\vee \otimes \Vcal)^\nonneg  - \deg(\Dcal^\vee \otimes \Fcal)^\nonneg \leq \deg(\Vcal^\vee \otimes \Vcal)^\nonneg.
\end{equation*}
In addition, we have
\begin{align*} 
\deg\left((\Vcal^{\leq \mu(\Ecal)})^\vee \otimes \Vcal\right)^\nonneg \leq \deg\left((\Vcal^{\leq \mu(\Ecal)})^\vee \otimes \Vcal\right)^\nonneg + \deg\left((\Vcal^{> \mu(\Ecal)})^\vee \otimes \Vcal\right)^\nonneg 
= \deg(\Vcal^\vee \otimes \Vcal)^\nonneg
\end{align*}
where equality holds if and only if $\Vcal^{> \mu(\Ecal)}$ is zero, which occurs precisely when $\Vcal$ and $\Ecal$ are isomorphic by the condition \ref{HN polygon inequality for extensions, aux}.
Hence it suffices to show
\begin{equation}\label{key inequality for step 2 alt form, reduced}
\begin{aligned}
\deg(\Dcal^\vee \otimes \Vcal)^\nonneg - \deg(\Dcal^\vee \otimes \Fcal)^\nonneg & \leq \deg\left((\Vcal^{< \mu(\Ecal)})^\vee \otimes \Vcal\right)^\nonneg.
\end{aligned}
\end{equation}

Let us write $\HNvec(\Dcal) := (d_i), \HNvec(\Fcal) := (f)$, and $\HNvec(\Vcal) := (v_j)$. By the condition \ref{HN polygon inequality for extensions, aux}, we have $f = \sum v_j - \sum d_i$. Then by Lemma \ref{geometric representation of degrees} and the condition \ref{slope dominance between kernel and image, aux} we find
\begin{align*}
\deg(\Dcal^\vee \otimes \Vcal)^\nonneg - \deg(\Dcal^\vee \otimes \Fcal)^\nonneg &= \sum_{\mu(d_i) \leq \mu(v_j)} d_i \times v_j - \sum d_i \times (\sum v_j - \sum d_i)\\
&= \sum_{\mu(d_i) \leq \mu(v_j)} d_i \times v_j - \sum d_i \times \sum v_j \\
&= - \sum_{\mu(d_i) > \mu(v_j)} d_i \times v_j = \sum_{ \mu(v_j) < \mu(d_i)} v_j \times d_i\\
&= \deg\left((\Vcal^{< \mu(\Ecal)})^\vee \otimes \Dcal \right)^\nonneg.
\end{align*}
Moreover, by Lemma \ref{nonneg degree for dominating HN polygons} and the condition \ref{dominant HN polygons for kernel} we have
\begin{equation*}\label{key inequality for step 2 alt form, reduced vector form} 
\begin{aligned}
\deg\left((\Vcal^{< \mu(\Ecal)})^\vee \otimes \Dcal \right)^\nonneg  & \leq \deg\left((\Vcal^{< \mu(\Ecal)})^\vee \otimes \Vcal\right)^\nonneg.
\end{aligned}
\end{equation*}
We thus deduce the desired inequality \eqref{key inequality for step 2 alt form, reduced}, thereby completing the proof. 
\end{proof}

\begin{lemma}\label{extension necessary conditions}
Let $\Dcal$, $\Ecal$, and $\Fcal$ be vector bundles on $\schff$ which fit into a short exact sequence
\[ 0 \longrightarrow \Dcal \longrightarrow \Ecal \longrightarrow \Fcal \longrightarrow 0.\]
\begin{enumerate}[label=(\arabic*)]
\item We have $\mumax(\Dcal) \leq \mumax(\Ecal)$, $\mumin(\Ecal) \leq \mumin(\Fcal)$, and $\HN(\Ecal) \leq \HN(\Dcal \oplus \Fcal)$. 
\smallskip

\item $\HN(\Dcal)$ lies below or on $\HN(\Ecal)$. 
\end{enumerate}
\end{lemma}

\begin{proof}
The nonzero maps $\Dcal \inj \Ecal$ and $\Ecal \surj \Fcal$ yield the first two inequalities by Lemma \ref{basic properties of stable bundles} and Proposition \ref{cohomologies of stable bundles}. The remaining assertions are standard consequences of the Harder-Narasimhan formalism, as noted by Kedlaya \cite[Lemma 3.4.15 and Lemma 3.4.17]{Kedlaya_arizona}. 
\end{proof}

\begin{remark}
In fact, the nonzero maps $\Dcal \inj \Ecal$ and $\Ecal \surj \Fcal$ yield much stronger conditions than the first two inequalities in Lemma \ref{extension necessary conditions}, as noted by the author in the previous works \cite[Theorem 1.1.2]{Hong_quotvb} and \cite[Theorem 1.2.1]{Hong_subvb}. 
\end{remark}

\begin{theorem}\label{main theorem}
Let $\Dcal, \Ecal$, and $\Fcal$ be vector bundles on $\schff$ such that $\Ecal$ is semistable. There exists a short exact sequence
\begin{equation}\label{main theorem short exact sequence}
0 \longrightarrow \Dcal \longrightarrow \Ecal \longrightarrow \Fcal \longrightarrow 0
\end{equation}
if and only if we have $\mumax(\Dcal) \leq \mu(\Ecal) \leq \mumin(\Fcal)$ and $\HN(\Ecal) \leq \HN(\Dcal \oplus \Fcal)$. 
\end{theorem}

\begin{proof}
The necessity part immediately follows from Lemma \ref{extension necessary conditions}. For the sufficiency part, we henceforth assume the inequalities $\mumax(\Dcal) \leq \mu(\Ecal) \leq \mumin(\Fcal)$ and $\HN(\Ecal) \leq \HN(\Dcal \oplus \Fcal)$. Let us write $r$ for the number of distinct slopes in $\HN(\Fcal)$ and proceed by induction on $r$. 

We first consider the base case where $\HN(\Fcal)$ is a line segment, which means by Theorem \ref{existence of HN decomp} that $\Fcal$ is semistable. Since all slopes of $\Fcal^\vee \otimes \Dcal$ are negative by Proposition \ref{basic properties of stable bundles}, we know by Proposition \ref{H1 diamond is locally spatial} that $\Extdiamond(\Fcal, \Dcal) \cong \Hcal^1(\Fcal^\vee \otimes \Dcal)$ is a locally spatial diamond over $\algclosedperfdfield$ with 
\begin{equation*}\label{dimension of Ext(F, D)_E for semistable F}
\dim \Extdiamond(\Fcal, \Dcal) = \deg(\Dcal^\vee \otimes \Fcal)^\nonneg.
\end{equation*}
Let $T$ be the set of all isomorphism classes of vector bundles $\Vcal$ on $\schff$ which fit into a short exact sequence
\[0 \longrightarrow \Dcal \longrightarrow \Vcal \longrightarrow \Fcal \longrightarrow 0.\] 
Proposition \ref{Ext(F, D)_E dimension formula}, Proposition \ref{key inequality for semistable quotient} and Lemma \ref{extension necessary conditions} together yield
\begin{align*}
\dim \Extdiamond(\Fcal, \Dcal)_\Vcal \leq \dim \Extdiamond(\Fcal, \Dcal) \quad\quad \text{ for each } \Vcal \in T
\end{align*}
where equality may hold only for $\Vcal = \Ecal$. Moreover, we have a decomposition
\[ |\Extdiamond(\Fcal, \Dcal)| = \bigsqcup_{\Vcal \in T} |\Extdiamond(\Fcal, \Dcal)_\Vcal|. \]
Since $T$ is a finite set by Lemma \ref{extension necessary conditions}, we find
\[\dim \Extdiamond(\Fcal, \Dcal) = \max_{\Vcal \in T} \,\dim \Extdiamond(\Fcal, \Dcal)_\Vcal\]
and consequently deduce that $T$ must contain the isomorphism class of $\Ecal$ as desired.

We now assume $r>1$ for the induction step. Let us write $\HNvec(\Dcal) := (d_i)$, $\HNvec(\Ecal):= (e)$ and $\HNvec(\Fcal) := (f_j)$. We find $\mu(f_r) \geq \mu(e) > \mu(f_r - \sum d_i)$ by our assumption on the HN polygons. 
Take $s$ to be the largest integer with $\displaystyle \mu\left(f_r +\sum_{i \leq s} d_i\right) \geq \mu(e)$, and set $\displaystyle e':= e - f_r - \sum_{i \leq s} d_i$. Define the vector bundles $\genred{\Dcal}$, $\Dcal'$, $\Ecal'$, $\genred{\Fcal}$, and $\Fcal'$ by
\[ \HNvec(\Dcal') = (d_i)_{i \leq s}, \quad \HNvec(\genred{\Dcal}) = (d_i)_{i >s}, \quad \HNvec(\Ecal') = (e'), \quad \HNvec(\genred{\Fcal}) = (f_j)_{j<r}, \quad \HNvec(\Fcal') = (f_r)\]
as illustrated in Figure \ref{induction illustration}. 
\begin{figure}[H]
\begin{tikzpicture}[scale=0.5]	
		\pgfmathsetmacro{\donex}{3}
		\pgfmathsetmacro{\doney}{0}
		\pgfmathsetmacro{\dtwox}{2.5}
		\pgfmathsetmacro{\dtwoy}{-2.5}
		\pgfmathsetmacro{\dthreex}{1}
		\pgfmathsetmacro{\dthreey}{-4.5}
		
		\pgfmathsetmacro{\fonex}{0.5}
		\pgfmathsetmacro{\foney}{3}	
		\pgfmathsetmacro{\ftwox}{1.5}
		\pgfmathsetmacro{\ftwoy}{3}
		\pgfmathsetmacro{\fthreex}{3}
		\pgfmathsetmacro{\fthreey}{2}
		
		\pgfmathsetmacro{\ex}{\donex+\dtwox+\dthreex+\fonex+\ftwox+\fthreex}
		\pgfmathsetmacro{\ey}{\doney+\dtwoy+\dthreey+\foney+\ftwoy+\fthreey}
				
		\coordinate (left) at (0, 0);
		\coordinate (p1) at (\fonex, \foney);
		\coordinate (p2) at (\fonex+\ftwox, \foney+\ftwoy);
		\coordinate (p3) at (\fonex+\ftwox+\fthreex, \foney+\ftwoy+\fthreey);

		\coordinate (q1) at (\ex-\dtwox-\dthreex, \ey-\dtwoy-\dthreey);
		\coordinate (q2) at (\ex-\dthreex, \ey-\dthreey);
		\coordinate (right) at (\ex, \ey);
		
		\draw[step=1cm,thick, color=green] (left) -- (right);
		\draw[step=1cm,thick, color=red] (left) -- (p1) -- (p2) -- (p3);
		\draw[step=1cm,thick, color=blue] (p3) -- (q1) -- (q2) -- (right);

		\draw [fill] (left) circle [radius=0.05];
		\draw [fill] (right) circle [radius=0.05];		
		
		\draw [fill] (p1) circle [radius=0.05];		
		\draw [fill] (p2) circle [radius=0.05];	
		\draw [fill] (p3) circle [radius=0.05];	

		\draw [fill] (q1) circle [radius=0.05];		
		\draw [fill] (q2) circle [radius=0.05];	
		
		\pgfmathsetmacro{\vex}{\ex*0.5}	
		\pgfmathsetmacro{\vey}{\ey*0.5}	

		\path (p2) ++(-0.7, -0.0) node {\color{red}$\Fcal$};
		\path (q2) ++(0.6, 0.0) node {\color{blue}$\Dcal$};
		\path (\vex, \vey) ++(0.0, -0.5) node {\color{green}$\Ecal$};
\end{tikzpicture}
\hspace{0.3cm}
\begin{tikzpicture}[scale=0.4]
        \pgfmathsetmacro{\textycoordinate}{5}
		\draw[->, line width=0.6pt] (0, \textycoordinate) -- (1.5,\textycoordinate);
		\draw (0,0) circle [radius=0.00];	
\end{tikzpicture}
\hspace{0.3cm}
\begin{tikzpicture}[scale=0.5]
		\pgfmathsetmacro{\donex}{3}
		\pgfmathsetmacro{\doney}{0}
		\pgfmathsetmacro{\dtwox}{2.5}
		\pgfmathsetmacro{\dtwoy}{-2.5}
		\pgfmathsetmacro{\dthreex}{1}
		\pgfmathsetmacro{\dthreey}{-4.5}
		
		\pgfmathsetmacro{\fonex}{0.5}
		\pgfmathsetmacro{\foney}{3}	
		\pgfmathsetmacro{\ftwox}{1.5}
		\pgfmathsetmacro{\ftwoy}{3}
		\pgfmathsetmacro{\fthreex}{3}
		\pgfmathsetmacro{\fthreey}{2}
		
		\pgfmathsetmacro{\ex}{\donex+\dtwox+\dthreex+\fonex+\ftwox+\fthreex}
		\pgfmathsetmacro{\ey}{\doney+\dtwoy+\dthreey+\foney+\ftwoy+\fthreey}
				
		\coordinate (left) at (0, 0);
		\coordinate (p1) at (\fonex, \foney);
		\coordinate (p2) at (\fonex+\ftwox, \foney+\ftwoy);
		\coordinate (p3) at (\fonex+\ftwox+\fthreex, \foney+\ftwoy+\fthreey);

		\coordinate (q1) at (\ex-\dtwox-\dthreex, \ey-\dtwoy-\dthreey);
		\coordinate (q2) at (\ex-\dthreex, \ey-\dthreey);
		\coordinate (right) at (\ex, \ey);
		
		\coordinate (r1) at (\fthreex, \fthreey);
		\coordinate (r2) at (\fthreex+\donex, \fthreey+\doney);
		\coordinate (r3) at (\fonex+\fthreex+\donex, \foney+\fthreey+\doney);
		
		\draw[step=1cm,thick, color=green] (left) -- (right);
		\draw[step=1cm,thick,dotted, color=red] (left) -- (p1) -- (p2) -- (p3);
		\draw[step=1cm,thick,dotted, color=blue] (p3) -- (q1) -- (q2) -- (right);
		\draw[step=1cm,thick, color=blue] (q1) -- (q2) -- (right);
		\draw[step=1cm,thick, color=orange] (left) -- (r1);
		\draw[step=1cm,thick, color=violet] (r1) -- (r2);
		\draw[step=1cm,thick, color=red] (r2) -- (r3) -- (q1);
		\draw[step=1cm,thick, color=teal] (r2) -- (right);

		\draw [fill] (left) circle [radius=0.05];
		\draw [fill] (right) circle [radius=0.05];		
		
		\draw [fill] (p1) circle [radius=0.05];		
		\draw [fill] (p2) circle [radius=0.05];	
		\draw [fill] (p3) circle [radius=0.05];	

		\draw [fill] (q1) circle [radius=0.05];		
		\draw [fill] (q2) circle [radius=0.05];	
		
		\draw [fill] (r1) circle [radius=0.05];		
		\draw [fill] (r2) circle [radius=0.05];	
		\draw [fill] (r3) circle [radius=0.05];		

		\pgfmathsetmacro{\vdprimex}{\fthreex+\donex*0.5}	
		\pgfmathsetmacro{\vdprimey}{\fthreey+\doney*0.5}		
		\pgfmathsetmacro{\veprimex}{(\fthreex+\donex+\ex)*0.5}	
		\pgfmathsetmacro{\veprimey}{(\fthreey+\doney+\ey)*0.5}	
		\pgfmathsetmacro{\vfprimex}{\fthreex*0.5}	
		\pgfmathsetmacro{\vfprimey}{\fthreey*0.5}			
		\pgfmathsetmacro{\vex}{\ex*0.5}	
		\pgfmathsetmacro{\vey}{\ey*0.5}	
		
		\path (r3) ++(-0.5, 0) node {\color{red}$\genred{\Fcal}$};
		\path (q2) ++(0.6, 0.0) node {\color{blue}$\genred{\Dcal}$};
		\path (\vex, \vey) ++(0.0, -0.5) node {\color{green}$\Ecal$};
		\path (\vdprimex, \vdprimey) ++(0.0, 0.5) node {\color{violet}$\Dcal'$};
		\path (\veprimex, \veprimey) ++(0.0, 0.5) node {\color{teal}$\Ecal'$};
		\path (\vfprimex, \vfprimey) ++(0.0, 0.5) node {\color{orange}$\Fcal'$};
\end{tikzpicture}
\caption{Construction of $\genred{\Dcal}$, $\Dcal'$, $\Ecal'$, $\genred{\Fcal}$, and $\Fcal'$}\label{induction illustration}
\end{figure}

By construction, we have
\[ \HN(\Ecal) \leq \HN(\Dcal' \oplus \Ecal' \oplus \Fcal') \quad \text{ and } \quad \mumax(\genred{\Dcal})<\mu(\Ecal') \leq \mu(\Ecal) \leq \mumin(\genred{\Fcal}).\]
Then by the induction hypothesis, we obtain short exact sequences
\[0 \longrightarrow \Dcal' \oplus \Ecal' \longrightarrow \Ecal \longrightarrow \Fcal' \longrightarrow 0 \quad \text{ and } \quad 0 \longrightarrow \genred{\Dcal} \longrightarrow \Ecal' \longrightarrow \genred{\Fcal} \longrightarrow 0.\]
These sequences together yield a commutative diagram of short exact sequences
\begin{equation*}
\begin{tikzcd}
0 \arrow[r]& \Dcal' \oplus \genred{\Dcal} \arrow[r, "\sim"]\arrow[d, hookrightarrow]& \Dcal \arrow[r]\arrow[d, "\alpha", hookrightarrow]& 0 \arrow[r]\arrow[d, hookrightarrow]& 0\\
0 \arrow[r]& \Dcal' \oplus \Ecal' \arrow[r]& \Ecal \arrow[r]& \Fcal' \arrow[r]& 0
\end{tikzcd}
\end{equation*}
which, by the snake lemma, gives rise to a short exact sequence
\[0 \longrightarrow \genred{\Fcal} \longrightarrow \coker(\alpha) \longrightarrow \Fcal' \longrightarrow 0.\]
Since this sequence is split by Proposition \ref{basic properties of stable bundles}, we obtain a desired exact sequence \eqref{main theorem short exact sequence}.
\end{proof}

\bibliographystyle{amsalpha}

\bibliography{Bibliography}
	
\end{document}